\documentclass[aap]{imsart}

\RequirePackage{amsthm,amsmath,amsfonts,amssymb}
\RequirePackage[numbers]{natbib}
\usepackage{enumerate}
\usepackage{enumitem}
\RequirePackage[colorlinks,citecolor=blue,urlcolor=blue]{hyperref}
\RequirePackage{graphicx}

\startlocaldefs
\theoremstyle{plain}

\newtheorem{theorem}{Theorem}[section]
\newtheorem{lemma}[theorem]{Lemma}

\newtheorem*{question}{Question}
\newtheorem{proposition}{Proposition}[section]
\newtheorem{corollary}{Corollary}[section]
\newtheorem{definition}{Definition}[section]

\numberwithin{equation}{section}
\DeclareMathOperator*{\reg}{reg}
\newcommand{\loc}{\text{\rm{loc}}}

\newcommand{\eps}{\varepsilon}

\newcommand{\lb}{\label}

\newcommand{\beq}{\begin{equation}}
\newcommand{\eeq}{\end{equation}}

\newcommand{\bal}{\begin{align}}
\newcommand{\eal}{\end{align}}
\newcommand{\bals}{\begin{align*}}
\newcommand{\eals}{\end{align*}}

\newcommand{\bbR}{{\mathbb{R}}}

\endlocaldefs

\begin{document}

\begin{frontmatter}
\title{McKean-Vlasov equations involving hitting times: blow-ups and global solvability}
\runtitle{McKean-Vlasov equations involving hitting times}

\begin{aug}
\author[A]{\fnms{Erhan} \snm{Bayraktar}\ead[label=e1]{erhan@umich.edu}},
\author[B]{\fnms{Gaoyue } \snm{Guo}\ead[label=e2]{gaoyue.guo@centralesupelec.fr}}
\author[C]{\fnms{Wenpin} \snm{Tang}\ead[label=e3]{wt2319@columbia.edu}}
\and
\author[D]{\fnms{Yuming Paul} \snm{Zhang}\ead[label=e4]{yzhangpaul@ucsd.edu}}
\address[A]{Department of Mathematics, University of Michigan,
\printead{e1}}

\address[B]{Laboratoire MICS, Universit\'e Paris-Saclay CentraleSup\'elec,
\printead{e2}}

\address[C]{Department of Industrial Engineering and Operations Research, Columbia University,
\printead{e3}}

\address[D]{Department of Mathematics, University of California, San Diego,
\printead{e4}}
\end{aug}

\begin{abstract}
This paper is concerned with the analysis of blow-ups for two McKean-Vlasov equations involving hitting times.
Let $(B(t); \, t \ge 0)$ be standard Brownian motion, and $\tau:= \inf\{t \ge 0: X(t) \le 0\}$ be the hitting time to zero of a given process $X$.
The first equation is $X(t) =X(0-) + B(t) - \alpha \mathbb{P}(\tau \le t)$.
We provide a simple condition on $\alpha$ and the distribution of $X(0-)$ such that the corresponding Fokker-Planck equation has no blow-up, and thus the McKean-Vlasov dynamics is well-defined for all time $t \ge 0$.
Our approach relies on a connection between the McKean-Vlasov equation and the supercooled Stefan problem, as well as several comparison principles.
The second equation is $X(t) =X(0-) + \beta t + B(t) + \alpha \ln \mathbb{P}(\tau >t)$, $t \ge 0$, whose Fokker-Planck equation is non-local.
We prove that for $\beta > 0$ sufficiently large and $\alpha$ no greater than a sufficiently small positive constant, there is no blow-up and the McKean-Vlasov dynamics is well-defined for all time $t \ge 0$.
The argument is based on a new transform, which removes the non-local term, followed by a relative entropy analysis.
\end{abstract}



\end{frontmatter}

\section{Introduction and main results}
Complex systems are central to the scientific modeling of real world phenomena. 
A challenge in mathematical modeling is to provide reasonably simple frameworks to capture the collective behaviors of individuals with intricate interactions. 
One famous example is the McKean-Vlasov equations, which were considered by Kac \cite{Kac57} in the context of statistical physics, and were further developed by McKean \cite{Mc66} to study weakly interacting particles.
The McKean-Vlasov equations have proved to be a powerful tool for modeling the mean field behavior of disordered systems, with applications including
the dynamics of granular media \cite{BC98, BGG13}, mathematical biology \cite{MC07, KS71}, economics and social networks \cite{CDL13, HMR06}, and deep neural networks \cite{MMN18, RV18}.
There have been a rich body of works on McKean-Vlasov equations, see \cite{CD1, CD2} for a detailed exposition.

In this paper, we are concerned with a class of generalized McKean-Vlasov equations which involve hitting times as boundary penalties.
These equations take the general form:
\begin{equation}
\label{eq:MVHG}
\displaystyle 
\left\{ \begin{array}{lcl}
X(t) = X(0-) + \beta t + B(t) + f(s(t)), \, t \ge 0, \\ 
\tau:= \inf\{t \ge 0: X(t) \le 0\}, \\
s(t): = \mathbb{P}(\tau \le t), 
\end{array}\right.
\end{equation}
where $X(0-)$ has a distribution supported on $(0, \infty)$, $\beta \in \mathbb{R}$ is the drift, 
$(B(t); \, t \ge 0)$ is standard Brownian motion, and $f: [0, 1) \to \mathbb{R}$ is a feedback function.
While our approaches may be used to study general cases,
we focus on the following two special scenarios.
\begin{enumerate}[itemsep = 3 pt]
\item
$\beta = 0$ and $f(x) = - \alpha x$, $\alpha > 0$: 
The equation \eqref{eq:MVHG} specializes to
\begin{equation}
\label{eq:MVHLin}
\displaystyle 
\left\{ \begin{array}{lcl}
X(t) =X(0-) + B(t) - \alpha \mathbb{P}(\tau \le t), \, t \ge 0, \\ 
\tau:= \inf\{t \ge 0: X(t) \le 0\}.
\end{array}\right.
\end{equation}
This model was originated in the study of the integrate-and-fire mechanism in neuroscience, from both probability aspects \cite{DI15, DI15b} and PDE perspectives \cite{CC11, CC13, CC15}.
It also arose as a toy model to study the mean field behavior of contagious financial networks \cite{HLS19}.
These works mainly dealt with the well-posedness of the McKean-Vlasov dynamics \eqref{eq:MVHLin}.
By letting $p(t, \cdot)$ be the sub-probability density of $X(t) 1_{\{\tau > t\}}$ and $N(t):= \partial_t \mathbb{P}(\tau \le t)$, 
the corresponding Fokker-Planck equation is 
\begin{equation}
\label{eq:MVHLinFP}
 \left\{ \begin{array}{lcl}
    p_t=\frac{1}{2}p_{xx}+\alpha N(t) p_x &&\text{ in }[0,T)\times (0,\infty),\\
    N(t)=\frac{1}{2}p_x(t,0),\quad p(t,0)=0 &&\text{ for }t\in [0,T),\\
    p(0,x)=p_{0}(x)&&\text{ for }x\in [0,\infty),
\end{array}\right.
\end{equation}
where $p_{0}(x)$ is the probability density of $X(0-)$.
As was shown in \cite{CC13}, the negative feedback $\alpha \le 0$ is classical, and there is a unique smooth solution to \eqref{eq:MVHLinFP} for all $t \ge 0$. 
The positive feedback $\alpha > 0$ is more subtle \cite{CC11, HLS19}: $s(t):=\mathbb{P}(\tau \le t)$ may not be absolutely continuous, and there may exist $T_{*} > 0$ such that $N(T_{*})= \infty$.
Such $T_{*}$ is called a blow-up, which is the main obstacle to analyze the Fokker-Planck equation \eqref{eq:MVHLinFP}, and study the well-posedness of the McKean-Vlasov dynamics \eqref{eq:MVHLin}.
To work around the blow-ups, \cite{DI15b} proposed the notion of a `physical solution' satisfying
$s(t) - s(t-) =\inf\{x \ge 0: \mathbb{P}(X(t-) \in (0,\alpha x) ) < x\}$ to the McKean-Vlasov dynamics \eqref{eq:MVHLin},
and proved the global existence by a particle system approximation.
A recent breakthrough \cite{DNS19} connected the Fokker-Planck equation \eqref{eq:MVHLinFP} to the supercooled Stefan problem, and as a byproduct the uniqueness of the physical solution is proved in the presence of blow-ups.
See also \cite{BS20, CSS20, LS18} for related developments.
\item
$f(x) = \alpha \ln (1-x)$, $\alpha > 0$: The equation \eqref{eq:MVHG} specializes to 
\begin{equation}
\label{eq:MVHLog}
\displaystyle 
\left\{ \begin{array}{lcl}
X(t) = X(0-) + \beta t + B(t) + \alpha \ln \mathbb{P}(\tau > t), \, t \ge 0, \\ 
\tau:= \inf\{t \ge 0: X(t) \le 0\}.
\end{array}\right.
\end{equation}
Similar to \cite{HLS19}, this model was proposed in \cite{NS19} to study the systemic risk of default financial networks. 
The nonlinearity of `$\ln$' comes from the assumption that after $k$ banks default at time $t$,
the value of each remaining bank is reduced by a factor of $\left(1 - \frac{k}{\# \mbox{\tiny banks at time } t-}\right)^{-\alpha}$.
Such a phenomenon is called a default cascade.
By letting $q(t, \cdot)$ be the sub-probability density of $X(t) 1_{\{\tau >t\}}$ and $\lambda(t): = \partial_t \ln \mathbb{P}(\tau > t)$, the corresponding Fokker-Planck equation is non-local:
\begin{equation}
\label{eq:MVHLogFP}
\displaystyle 
 \left\{ \begin{array}{lcl}
q_t=\frac{1}{2}q_{xx}-( \alpha \lambda(t)+ \beta)q_x &&\text{ in }[0,T)\times(0,\infty),\\
\lambda(t)= -\frac{1}{2}\frac{q_x(t,0)}{\int_0^\infty q(t,y)dy},\quad q(t,0)=0 &&\text{ on }[0,T),\\
q(0,x)=q_{0}(x)&&\text{ on }[0,\infty),
\end{array}\right.
\end{equation}
where $q_{0}(x)$ is the probability density of $X(0-)$.
Note that the equation \eqref{eq:MVHLogFP} is slightly different from that in \cite[Section 3]{NS19}, since we define $\lambda(t)$ as  $\partial_t \ln \mathbb{P}(\tau > t)$ instead of $\alpha \partial_t \ln \mathbb{P}(\tau > t)$.
There may also exist a blow-up $T_{*}$ such that $\|\lambda\|_{L^2[0,T_{*}]}:= \int_0^{T_{*}} \lambda^2(t) dt = \infty$.
It was proved in \cite{NS19} that a physical solution exists for all time $t$.
The uniqueness is still open, though it is believable that the arguments in \cite{DNS19} carry over to this setting.
\end{enumerate}

As we have already seen, the main difficulty in analyzing the McKean-Vlasov dynamics involving hitting times arises from the blow-ups.
Though \cite{DI15b, DNS19, NS19} proposed the physical solution to overcome this problem, it is still interesting to know whether there is a blow-up or not.
This is because the existence of a blow up implies a possible systemic risk event in a financial network. 
To be more precise, we ask the following question.
\begin{question}
Under what conditions on the distribution of $X_{0-}$ is there no blow-up in the Fokker-Planck equations \eqref{eq:MVHLinFP} and \eqref{eq:MVHLogFP} respectively?
\end{question}
If there is no blow-up, the physical solution coincides with a smooth solution possibly in some weak sense.
We simply say that the McKean-Vlasov dynamics is (well-)defined for all $t \ge 0$
if the corresponding Fokker-Planck equation does not exhibit blow-ups,
and is hence defined for all time in the classical sense.
In \cite{DI15, HLS19}, no blow-up conditions for the Fokker-Planck equation \eqref{eq:MVHLinFP} have been studied,
which assure that the McKean-Vlasov dynamics \eqref{eq:MVHLin} is defined for all time $t \ge 0$.
However, these conditions seem to be obscure, and are not easy to check.
No blow-up conditions for the Fokker-Planck equation \eqref{eq:MVHLogFP} have yet been explored, and it was conjectured in \cite{NS19} that there is no blow-up if $\alpha$ is sufficiently small. 

In this paper, we provide a simple criterion on the distribution of $X_{0-}$ under which the Fokker-Planck equation \eqref{eq:MVHLinFP} does not have any blow-up. 
We also study the problem of blow-ups for the Fokker-Planck equation \eqref{eq:MVHLogFP}, resolving the aforementioned conjecture.
To state the results, we need the following definition of weak and generalized solutions to \eqref{eq:MVHLinFP} and \eqref{eq:MVHLogFP}, respectively.
Below $L^1_{loc}([0,T))$ (resp. $L^2_{loc}([0,T))$) denotes the functions that are locally uniformly $L^1$ (resp. $L^2$) in $[0,T)$, 
and $L^\infty([0,T); L^1(\bbR^+))$ denotes 
\[
\left\{f: [0,T) \times \mathbb{R} \to \mathbb{R}^+ \,  \bigg| \,  f(t,\cdot)\in L^1(\bbR^+),\, \sup_{t\in [0,T)}   \|f(t,\cdot)\|_{L^1(\bbR^+)}<\infty\right\},
\]
and $W_2^{1,2}([0,T] \times [0,\infty))$ denotes the Sobolev space $L^2([0,T] \times [0,\infty))$
whose first weak derivative in time and the first two weak derivatives in space belong to $L^2([0,T] \times [0,\infty))$, equipped with the associated Sobolev norm.

\smallskip
\begin{definition}\label{def sol}
~
\begin{enumerate}[itemsep = 3 pt]
\item 
A pair of functions $(p,N)$ is a weak solution to the Fokker-Planck equation \eqref{eq:MVHLinFP} in the time interval $[0,T)$ if 
\[
p\in 
L^\infty([0,T); L^1(\bbR^+)), \quad N\in L^1_{loc}([0,T)),
\]
$p,N$ are non-negative,
and for any test function $\phi(t,x)\in C^\infty([0,T']\times [0,\infty))$ with $T'<T$ such that $\phi,\phi_t,\phi_x,\phi_{xx}\in L^\infty([0,T']\times [0,\infty))$, we have
\begin{align*}
    \int_0^{T'}\int_0^\infty & p(t,x)\left[-\phi_t(t,x)+\alpha N(t)\phi_x(t,x)-\frac{1}{2}\phi_{xx}(t,x)\right]dxdt\\
    =&-\int_0^{T'} N(t)\phi(t,0)dt+\int_0^\infty p_{0}(x)\phi(0,x)dx-\int_0^\infty p(T',x)\phi(T',x)dx.
\end{align*}
\item \cite{LSU68, NS19}
A pair of functions $(q,\lambda)$ is a generalized solution to the Fokker-Planck equation \eqref{eq:MVHLogFP} in the time interval $[0,T)$ if 
\[
 \lambda \in L^2_{loc}([0,T)),
\]
and for $T' < T$, the unique solution to the equation
\[
q_t=\frac{1}{2}q_{xx}-(\alpha \lambda(t)+\beta)q_x, \quad q(t,0)=0, \quad q(0,x)=q_{0}(x),
\]
in $W_2^{1,2}([0, T'] \times [0, \infty))$ satisfies $\lambda(t)=-\frac{1}{2}\frac{q_x(t,0)}{\int_0^\infty q(t,y)dy}$ for almost every $t \in [0, T']$.
\end{enumerate}
\end{definition}

In the above definition, a generalized solution to the equation \eqref{eq:MVHLogFP} is more restrictive than a weak solution to the equation \eqref{eq:MVHLinFP}, and it requires the uniqueness along with the existence.
This complication is due to the non-local term `$\int_0^\infty q(t,y)dy$' in the Fokker-Planck equation \eqref{eq:MVHLogFP}.
Our first result provides sufficient conditions for the Fokker-Planck equations \eqref{eq:MVHLinFP} and \eqref{eq:MVHLogFP} to exhibit blow-ups, which is in a similar spirit to \cite[Theorem 2.2]{CC11}.

\smallskip
\begin{proposition}\lb{T.1.2}
~
\begin{enumerate}[itemsep = 3 pt]
\item
If there exists $\mu>0$ such that 
\beq\lb{1.10}
\mu\alpha\int_0^\infty e^{-\mu x}p_0(x)dx\geq 1,
\eeq
then there is no weak solution to the Fokker-Planck equation \eqref{eq:MVHLinFP} for all time $t \ge 0$.
In this case, the solution can only exist before time
\begin{equation}
\label{eq:Tbu1}
T:=\frac{2}{\mu^2}\ln \frac{\int_0^\infty p_0(x)dx}{\int_0^\infty e^{-\mu x}p_0(x)dx}.
\end{equation}
In particular, if $\alpha > \min_{\mu > 0} \{(\mu \int_0^\infty e^{-\mu x}p_0(x)dx)^{-1},2\int_0^\infty xp_0(x)dx \}$,
then there is no weak solution to the Fokker-Planck equation \eqref{eq:MVHLinFP} for all time $t \ge 0$.
\item
If there exists a positive number $\mu>2\beta$ such that 
\beq\lb{4.3}
(1+\alpha\mu)\int_0^\infty e^{-\mu x}q_0(x)dx\geq \int_0^\infty q_0(x)dx,
\eeq
then there is no generalized solution to the Fokker-Planck equation \eqref{eq:MVHLogFP} for all time $t \ge 0$.
In this case, the solution can only exist before time 
\begin{equation}
\label{eq:Tbu2}
T: =\frac{2}{\mu(\mu-2\beta)}\ln\left(\frac{\int_0^\infty q_0(x)dx}{\int_0^\infty e^{-\mu x}q_0(x)dx}\right).
\end{equation}
\end{enumerate}
\end{proposition}

In view of Lemma \ref{l.clascsol}, if we further assume
\beq\lb{c.3.3}
\limsup_{x\to 0+}p_0(x)<\frac{1}{\alpha}\quad\text{and}\quad\lim_{x\to\infty } p_0(x)=0,
\eeq
then the solution blows up ($N(t)\to\infty$) before time $T$ with $T$ defined by \eqref{eq:Tbu1}.
By Lemma \ref{T.5.1}, if we assume that $q_0(\cdot)\in W_2^1([0,\infty))$ and $q_0(0)=0$, then there exists a time $t_{\reg}\in (0,T]$ with $T$ defined by \eqref{eq:Tbu2} such that $\lim_{t\uparrow t_{\reg}}\|\lambda\|_{L^2[0,t]}=\infty$.
We emphasize that the physical solution allows the presence of blow ups, and Proposition \ref{T.1.2} then implies that the first blow up must occur before time $T$ in each Fokker-Planck equation.

The next theorem, which is our main result, gives a simple condition under which there is no blow-up for the Fokker-Planck equation \eqref{eq:MVHLinFP}.
Thus, a weak solution to \eqref{eq:MVHLinFP}, which is proved to be a classical one, 
is defined for all time $t \ge 0$.
Consequently, the McKean-Vlasov dynamics \eqref{eq:MVHLin} is defined for all time $t \ge 0$.
\begin{theorem}\lb{T.4.1}
Let $p_0$ be a probability density supported on $(0,\infty)$ such that 
\eqref{c.3.3} holds
and 
\beq\lb{c.3.3'}
\int_0^x (1-\alpha p_0(y))dy>0, \quad \mbox{for all } x \in (0,\infty).
\eeq
Assume that the weak solution $(p,N)$ to the Fokker-Planck equation \eqref{eq:MVHLinFP} 
with initial data $p_0$ exists for a short time. Then it exists for all time $t\geq 0$.
Moreover, $(p,N) $ satisfies the equation in the classical sense, and for all $t>0$, and for some $C>0$ depending only on $p_0$ we have
\[
N(t)\leq C\alpha^{-1}(1+\alpha^{-1}+(1+\alpha^\frac12)t^{-\frac{1}{2}}+(\ln t)^2)  
, \quad \mbox{for all } t>0.
\]
\end{theorem}

Note that under the assumption \eqref{c.3.3}, if we further assume the probability density $p_0$ to be piecewise continuous, then the weak solution $(p,N)$ exists for a short time. This can be done by the argument in \cite[Theorem 1.3]{classical8}.

Note that the condition \eqref{c.3.3'} is weaker than $\alpha <\|p_0\|_{\infty}^{-1}$. 
Indeed, we do not assume any $L^\infty$ bound on the initial data.
It was proved in \cite[Theorem 2.2]{LS20} that if $\alpha <\|p_0\|_{\infty}^{-1}$, the McKean-Vlasov dynamics \eqref{eq:MVHLin} is pathwise unique. 
Combining this result with Theorem \ref{T.4.1}, we get the following corollary.
\begin{corollary}
Assume that $\lim_{x \to \infty} p_0(x) = 0$ and $\alpha < \|p_0\|_{\infty}^{-1}$.
Then the McKean-Vlasov dynamics \eqref{eq:MVHLin} is defined for all time $t \ge 0$, and is pathwise unique.
\end{corollary}

Furthermore, we consider the Fokker-Planck equation \eqref{eq:MVHLinFP} with initial data of form $\delta_{x_0}$ (delta mass).
Applying Proposition \ref{T.1.2} (1) and Theorem \ref{T.4.1} yields the following corollary. 
\begin{corollary}\lb{T.3.3}
Let $\alpha,x_0>0$.
Let $(p(\cdot,x; x_0),N(\cdot\,; x_0))$ be a weak solution to the Fokker-Planck equation \eqref{eq:MVHLinFP} with initial data $\delta_{x_0}$, and assume that $(p(\cdot,x; x_0),N(\cdot\,; x_0))$ exists for a small time. 
Then
\begin{itemize}[itemsep = 3 pt]
\item
if $ {\alpha}<x_0$, the solution $p(t, \cdot; x_0)$ exists for all time $t \ge 0$ and $N(t;x_0)<\infty$ for all $t>0$.
\item
if $\alpha>2x_0$, the solution cannot exist for all time. 
Moreover, there exists $T_{x_0} > 0$ such that $\limsup_{t\to T_{x_0}}p_x(t, \cdot; x_0)=\infty$.
\end{itemize}
\end{corollary}

Now we turn to the Fokker-Planck equation \eqref{eq:MVHLogFP}. 
Due to the non-local term, it seems to be difficult to get a simple criterion for no blow-up.
Nevertheless, we are able to show that for any initial data if $\beta > 0$ is sufficiently large, and $\alpha$ is no greater than a sufficiently small positive constant, then a generalized solution to \eqref{eq:MVHLogFP} and the McKean-Vlasov dynamics \eqref{eq:MVHLog} is well-defined for all time $t \ge 0$.
This confirms a conjecture in \cite[Remark 2.8]{NS19}.
Below $W_2^{1}([0,\infty))$ denotes the Sobolev space of $L^2([0,\infty))$ functions whose first weak derivative belongs to $L^2([0,\infty))$.

\begin{theorem}\lb{T.3.2}
Let $q_0(\cdot)\in W^1_2([0,\infty))$ with $q_0(0)=0$, and assume that $q_0^2(x)/x$ is integrable on $(0,1)$. 
There exists $C_0>0$ depending only on $q_0$ such that if $\beta\geq C_0$ and $\alpha\leq \frac{1}{C_0}$, 
then the generalized solution $(q,\lambda)$ to the Fokker-Planck equation \eqref{eq:MVHLogFP} with initial data $q_0$ 
exists for all time $t \ge 0$.
Moreover, for some $C>0$ depending only on $q_0$ we have
\begin{equation}
\label{eq:entbd}
\int_0^t |\lambda(s)|^2 ds \leq C(1+t), \quad \mbox{for all } t>0.
\end{equation}
\end{theorem}

Note that \cite{NS19} only considers the case $\alpha > 0$, while Theorem \ref{T.3.2} extends to 
all $\alpha \le 0$ provided that $\beta>0$ is sufficiently large.
As a consequence of \eqref{eq:entbd}, there exists $C' > 0$ such that
$\mathbb{P}(\tau > t) \ge \exp(-C't)$ for all $t > 0$,
which gives a lower bound on the tail of the hitting time of the McKean-Vlasov dynamics \eqref{eq:MVHLog}.
The problem of the uniqueness is more subtle.
In our forthcoming paper \cite{BGTZ2}, we prove that for some initial distribution $q_0$, the McKean-Vlasov dynamics \eqref{eq:MVHLog} for large $\beta$ and small $\alpha$ is not unique in distribution.

The main idea to prove Theorem \ref{T.4.1} and Theorem \ref{T.3.2} consists of comparing the solution to  \eqref{eq:MVHLinFP} with the self-similar solution to the super-cooled Stefan problem, and comparing the solution to \eqref{eq:MVHLogFP} with the stationary solution to a transformed equation. 
In contrast with the fixed-point method used in \cite{DI15, HLS19, NS19}, we rely on comparison principles and relative entropy arguments which are of independent interest.

\bigskip
{\bf Organization of the paper}: 
In Section \ref{sc2}, we consider the weak solutions to the equation \eqref{eq:MVHLinFP}.
There we prove Theorem \ref{T.4.1}.
Section \ref{sc3} is devoted to the study of the equation \eqref{eq:MVHLogFP}, and Theorem \ref{T.3.2} is proved.

\section{Solutions to the Fokker-Planck equation \eqref{eq:MVHLinFP}}
\label{sc2}

In this section, we study the weak solutions to the Fokker-Planck equation \eqref{eq:MVHLinFP}. 
The idea of the proof is inspired from \cite[Theorem 4.1]{classical8} regarding the supercooled Stefan problem.
Using the transformation \eqref{2.2} below, it can be deduced from \cite{classical8} that if 
\[
\limsup_{x\to 0+}p_0(x)<\frac{1}{\alpha},
\]
the solution $p$ to the equation \eqref{eq:MVHLinFP} exists for a short time, and within the short time $N(t)<\infty$ for $t>0$. 
In comparison with \cite{classical8}, our approach is notably different:
(1) our existence and regularity results hold for all time;
(2) we consider an unbounded domain. 
To achieve these, technically, we need more delicate comparison principles (Lemma \ref{L.2.1} and Lemma \ref{T.3.1}).

Section \ref{sc21} presents preliminaries on the super-cooled Stefan problem, and its self-similar solution. 
In Section \ref{sc22}, we provide key comparison lemmas which will be used in the proof of Theorem \ref{T.4.1}.
Theorem \ref{T.4.1} and Corollary \ref{T.3.3} will be proved in Section \ref{sc23}.

\subsection{Preliminaries}
\label{sc21}
We also need the notion of classical solutions to the equation \eqref{eq:MVHLinFP}.  
To this end, we assume that the initial data $p_0$ satsifies
\beq\lb{1.2}
\begin{aligned}
    &p_0\in C^1(\bbR^+)\cap C([0,\infty))\cap L^1(\bbR^+),\\
    & p_0(0)=\lim_{x\to\infty}p_0(x)=\lim_{x\to\infty}\partial_xp_0(x)=0.
\end{aligned}
\eeq
\begin{definition}
A pair of functions $(p,N)$ is a classical solution to the Fokker-Planck equation \eqref{eq:MVHLinFP} in the time interval $[0,T)$ for a given $T \in (0, \infty]$ and with initial data $p_0$ satisfying \eqref{1.2}, if the following conditions are satisfied:
\begin{enumerate}[itemsep = 3 pt]
    \item $N(t)$ is a continuous function for all $t\in [0,T)$,   
    \item $p$ is continuous in $[0,T)\times[0,\infty)$, $p\in C^1_tC^2_x((0,T)\times \bbR^+)$, and for $t\in (0,T)$, $p(t,0^+)$ is well-defined and $p,p_x(t,x)\to 0$ as $x\to\infty$,
        \item The equation \eqref{eq:MVHLinFP} is satisfied in the classical sense.  
\end{enumerate}
\end{definition}

The following result is a simple variant of \cite[Theorem 3.1, Theorem 4.2]{CC13}.
\begin{lemma}\lb{T.1.1}{\rm (\cite{CC13})}
Let $p_0$ satisfy \eqref{1.2}. Then there exists a unique classical solution to the Fokker-Planck equation \eqref{eq:MVHLinFP} in the time interval $[0,T_*)$ for some $T_*>0$. 
The maximal time of existence $T_*>0$ is characterized as
\[T_*=\sup\{t>0: N(t)<\infty\}.\]
\end{lemma}

The lemma can be proved via a fixed point argument using that 
\[
\Gamma(N)(t):=2\int_{0}^\infty G(t,s(t)-x)\partial_x p_0(x)dx+\int_0^t N(\tau)G_x(t-\tau,s(t)-s(\tau))d\tau,
\]
which is derived through Green's identity,
defines a contraction mapping $\Gamma$ on the space of $\{N\in C([0,T]): \|N\|_\infty\leq M\}$ for some $M>0$ when $T$ is sufficiently small, see e.g. \cite{CC13, friedman1959free}. Here $s(t):=\int_0^tN(\tau)d\tau$ and 
\beq\lb{fund}
G(t,x):=\frac{1}{\sqrt{2\pi t}} e^{-{|x|^2}/{2t}},
\eeq
is the Green function of the heat equation on the real line. After finding out $N(t)$, $p(t,x)$ can be solved from the first equation in \eqref{eq:MVHLinFP} on $[0,T]$. The solution can then be extended up to the time of the first blow-up, see Theorem 4.2 \cite{CC13}.

Now we give a proof of Proposition \ref{T.1.2} (1).
\begin{proof}[Proof of Proposition \ref{T.1.2} (1)]
Suppose by contradiction that a weak solution exists for all time. From the weak formulation in Definition \ref{def sol} (1), taking $\phi(t,x)=e^{-\mu x}$ for some $\mu>0$, we have
\beq\lb{1.3}
\begin{aligned}
&\int_0^\infty e^{-\mu x}p(t,x)dx-\int_0^\infty e^{-\mu x}p_0(x)dx\\
&\qquad =\int_0^t\int_0^\infty \frac{\mu^2}2 e^{-\mu x} p(\tau,x)+\alpha\mu\, e^{-\mu x} N(\tau)p(\tau,x)dxd\tau-\int_0^t N(\tau)d\tau 
\end{aligned}
\eeq
Since $p\geq 0$, this yields
\[
   \int_0^\infty e^{-\mu x}p(t,x)dx-\int_0^\infty e^{-\mu x}p_0(x)dx\geq \int_0^tN(\tau)\left(\alpha\mu\int_0^\infty e^{-\mu x}p(\tau,x)dx-1\right)d\tau.\]
Writing $M_\mu(t):=\mu \alpha \int_0^\infty e^{-\mu x}p(t,x)dx-1$, we get 
\[
M_\mu(t)-M_\mu(0)\geq {\mu \alpha}\int_0^tN(\tau)M_\mu(\tau)d\tau.
\]
By Gronwall's inequality (see e.g., \cite[Theorem 2.4.5]{GW}), using
$
M_\mu(0)=\mu\alpha\int_0^\infty e^{-\mu x}p_0(x)dx-1\geq 0
$
by \eqref{1.10}, we have for all $t\geq 0$ that
\[
M_\mu(t)=\mu\alpha\int_0^\infty e^{-\mu x}p(t,x)dx- 1\geq 0.
\]
Therefore by \eqref{1.3} again, we obtain
\[
 \int_0^\infty e^{-\mu x}p(t,x)dx-\int_0^\infty e^{-\mu x}p_0(x)dx\geq \frac{\mu^2}{2}\int_0^t\int_0^\infty e^{-\mu x}p(\tau,x)dxds.
\]
This implies that 
\[
\int_0^\infty e^{-\mu x}p(t,x)dx\geq e^{\frac{\mu^2}{2} t}\int_0^\infty e^{-\mu x}p_0(x)dx\to \infty \text{ as }t\to\infty,
\]
which is impossible. In fact, $e^{\frac{\mu^2}{2} t}\int_0^\infty e^{-\mu x}p_0(x)dx$ cannot be greater or equal than $\int_0^\infty p_0(x)dx$ since otherwise
\[
\int_0^\infty p(t,x)dx> \int_0^\infty e^{-\mu x}p(t,x)dx\geq  \int_0^\infty p_0(x)dx,
\]
but setting $\phi\equiv 1$ (in the domain of $[0,T']\times [0,\infty)$) in the weak formulation in Definition \ref{def sol} (1) reveals that the total mass of $p$ is non-increasing in time.
Solving $e^{{\mu^2} T/2}\int_0^\infty e^{-\mu x}p_0(x)dx=\int_0^\infty p_0(x)dx$ for $T$ gives the upper bound on the existing time of solutions when \eqref{1.10} holds.
The second part follows from \cite[Theorem 1.1]{HLS19}.
\end{proof}

\smallskip
{\bf Super-cooled Stefan problem}: 
Let $(p,N)$ be a classical solution to the Fokker-Planck equation \eqref{eq:MVHLinFP} in the time interval $[0,T)$.
It is well-known that the transformation 
\beq\lb{2.2}
u(t,x):=p(t,x-\alpha s(t)),\quad s(t):=\int_0^tN(\tau)d\tau
\eeq
turns the equation into supercooled Stefan problem:
\beq\lb{2.7}
\left\{ \begin{array}{lcl}
u_t=\frac{1}{2}u_{xx} &&\text{ in }\{(t,x): x>\alpha s(t),\, t\in (0,T)\},\\
s'(t)=\frac{1}{2}u_x(t,\alpha s(t)),\quad u(t,\alpha s(t))=0 &&\text{ for }t\in [0,T),\\
u(0,x)=p_0(x-\alpha s(0))&&\text{ for }x\in [\alpha s(0),\infty).
\end{array}\right.
\eeq

We start with the following simple lemma of a comparison principle between two solutions with sub-quadratic growth in unbounded space-time domains. 

\begin{lemma}\lb{l.sim}
Let $s_*:[0,T]\to [0,\infty)$ be a continuous function. Suppose $u_t\leq \frac{1}{2}u_{xx}$ in $D_T:=\{(t,x): x> s_*(t),\, t\in [0,T)\}$ in the sense of distribution, i.e. for any smooth, non-negative $\varphi$ that is compactly supported in $D_T$ we have
\[
\int_{D_T} u(t,x)\left[\varphi_t(t,x)+\frac{1}{2}\varphi_{xx}(t,x)\right]dxdt+\int_{s_*(0)}^\infty u_{0}(x)\varphi(0,x)dx\geq 0.
\]
Let $v\geq 0$ satisfy $v_t= \frac{1}{2}v_{xx}$ in $D_T$. Then if $u(0,\cdot)\leq v(0,\cdot)$ in $(s_*(0),\infty)$,
\[
\limsup_{x\to\infty}\sup_{t\in [0,T)}\left\{\frac{u(t,x)-v(t,x)}{x^2}\right\}\leq 0\quad \text{ and }\quad \sup_{t\in [0,T)} \limsup_{x\to s_*(t)} \{u(t,x)-v(t,x)\}\leq 0,
\]
we have $u\leq v$ in $D_T$.
\end{lemma}
\begin{proof}
For any $\eps>0$, define
\[
v_\eps(t,x):=v(t,x)+\eps t+\eps x^2.
\]
By the assumption, for all $M=M(\eps)>0$ large enough we have $v_\eps(t,M)\geq u(t, M)$ for all $t\in [0,T)$.
Since $v_\eps$ satisfies the heat equation, and $v_\eps(0,\cdot)\geq u(0,\cdot)$, we can apply the comparison principle (see e.g., \cite[Corollary 6.26]{Lieberman}) in $\{(t,x)\in D_T: x<M\}$ to conclude that $v_\eps\geq {u}$ in $\{(t,x)\in D_T: x<M\}$. Passing $M\to\infty$ yields $v_\eps\geq {u}$ in $D_T$. Then passing $\eps\to 0$ yields $v\geq {u}$ in $D_T$.
\end{proof}

\begin{lemma}\label{l.clascsol}
Let $p_0$ be a probability density supported on $(0,\infty)$ that satisfies \eqref{c.3.3}.
Suppose that a weak solution $(p,N)$ to \eqref{eq:MVHLinFP} exists for a short time. Then there exists $T_*>0$ such that $(p,N)$ can be extended to $[0,T_*)$ and $(p,N)$ is a classical solution for $t\in (0,T_*)$.
The maximal time of existence $T_*>0$ is characterized as
\[
T_*=\sup\{t>0: N(t)<\infty\}.
\]

\end{lemma}
\begin{proof}
Suppose $(p,N)$ is a weak solution to \eqref{eq:MVHLinFP} in $[0,T)$ for some $T>0$.
Let $u(t,x)=p(t,x-\alpha s(t))$, and so $u$ is supported in $D_T:=\{(t,x): x>\alpha s(t),\, t\in [0,T)\}$. 
Since $0\leq N(\cdot)\in L^1_{\loc}$,  $s(t)$ is a continuous, non-decreasing function satisfying $s(0)=0$. 
We extend $u$ by $0$ to $[0,T)\times \bbR $. 
By Definition \ref{def sol} (1),
we get for any test function $\varphi:[0,T']\times\bbR\to\bbR$ with $T'<T$, that is smooth and bounded,
\[
\begin{aligned}
\int_0^\infty u(T',x)\varphi(T',x)dx=    \int_0^{T'}\int_0^\infty & u(t,x)\left[\varphi_t(t,x)+\frac{1}{2}\varphi_{xx}(t,x)\right]dxdt\\
    & -\int_0^{T'} N(t)\varphi(t,\alpha s(t))dt+\int_0^\infty u_{0}(x)\varphi(0,x)dx.
\end{aligned}
\]
Thus $u$ solves the heat equation in the interior of $D_T$. Moreover, since $N\geq 0$, we have $u_t\leq \frac12u_{xx}$ in $(0,T)\times[0,\infty)$ in the sense of distribution (see Lemma \ref{l.sim} with $s_*\equiv 0$).
 
Since $\sup_{t\in [0,T)}\|u(t,\cdot)\|_{L^1((\alpha s(t),\infty))}<\infty$,
by the hypothesis on $p$, 
$$\eta(t,x):=\int_{\alpha s(t)}^x u(t,y)dy+s(t),$$ is uniformly bounded in $D_T$. 
The equation yields that $\eta(t,x)$ solves the heat equation in $D_T$ in the sense of distribution, and $\eta$ is continuous in both space and time. Then, after restricting $\eta$ to $(0,T)\times [M,\infty)$ for some $M> \alpha s(T)+1$, we can obtain an explicit representation formula for $\eta$. 
Indeed, using the classical reflection method for heat equation with a source on half-line for $\eta(t,x+M)-\eta(t,M)e^{-{x}}$ (see \cite[Section 4.1]{strauss}), we get for all $t\in (0,T)$ and $x>0$,
\begin{equation}
\label{eq:repre}
\begin{aligned}
&\eta(t,x+M)=\eta(t,M)e^{-x}\\
&\qquad+\int_0^\infty\left(G(t,x-y)-G(t,x+y)\right)(\eta(0,y+M)-\eta(0,M)e^{-y})dy\\
&\qquad+\int_0^t\int_0^\infty \left(G(t-\tau,x-y)-G(t-\tau,x+y)\right)\left(\frac{1}{2}\eta(\tau,M)e^{-y}-\eta_t(\tau,M)e^{-y}\right)dyd\tau\\
&\qquad\qquad\quad=\int_0^\infty\left(G(t,x-y)-G(t,x+y)\right)\eta(0,y+M)dy\\
&\qquad-\int_0^\infty \left(G_t(t,x-y)-G_t(t,x+y)\right)\eta(0,M)e^{-y}dy\\
&\qquad+\int_0^t\int_0^\infty \left(G(t-\tau,x-y)-G(t-\tau,x+y)\right)\frac{1}{2}\eta(\tau,M)e^{-y}dyd\tau
\end{aligned}
\end{equation}
where $G$ is the Green function given by \eqref{fund}, and the integral involving $\eta_t$ is justified by integration by parts. 
Using the formula \eqref{eq:repre}, and $\lim_{x\to\infty}\eta_x(0,x)=\lim_{x\to\infty}p_0(x)=0$ by the assumption, we obtain
\beq\lb{aa1}
\lim_{x\to\infty}\sup_{t\in [0,T)}u(t,x)=\lim_{x\to\infty}\sup_{t\in [0,T)}\eta_x(t,x)=0.
\eeq
Moreover, it follows from the formula \eqref{eq:repre} (or from the parabolic interior regularity theory \cite[Theorem 11.5]{Lieberman} and \eqref{aa1}) that
\beq\lb{55555}
\lim_{x\to\infty} \partial_x u(t,x)=0\quad\text{ for all $t\in (0,T)$}.
\eeq

Now set
\beq\lb{utilde}
\tilde{u}(t,x):=\int_0^\infty (G(t,x-y)-G(t,x+y))p_0(y)dy
\eeq
which is then the solution to the heat equation in $(0,T)\times\bbR^+$ satisfying $ \tilde{u}(t,0)=0$ and $ \tilde{u}(0,\cdot)=p_0(\cdot)$. The assumption on $p_0$ yields $\lim_{x\to\infty}\sup_{t\in [0,T)}\tilde u(t,x)=0$.
 Due to \eqref{aa1}, it follows from Lemma \ref{l.sim} that $u\leq \tilde{u}$ in $[0,T)\times \bbR$. 
Due to  the first condition in \eqref{c.3.3}, there exists $\delta>0$ and $h>0$ such that $\tilde{u}<\frac{1}{\alpha}$ for all $x\in (0,h)$ and $t\in [0,\delta)$. 
Thus, the same inequality holds for $u$ in place of $\tilde{u}$. 
By taking $\delta$ to be small, we can assume that $h>s(\delta)$. Then, since $u$ is continuous for $t>0$, for some $M$ such that $M>s(\delta)+1$, \cite[Theorem 1.3]{classical8} yields that the following equation possesses a unique classical solution for a short time. For any $\eps\in (0,\delta)$,
\[
\left\{ \begin{array}{lcl}
\bar{u}_t=\frac{1}{2}\bar{u}_{xx} &&\text{ in }\{(t,x): x\in (\alpha \bar{s}(t),M),\, t>0\},\\
\bar{s}'(t)=\frac{1}{2}\bar{u}_x(t,\alpha \bar{s}(t)), \bar{s}(0)=s(\eps)&&\text{ for }t>0,\\
\bar{u}(t,\alpha \bar{s}(t))=0, \bar{u}(t,M)=g_\eps(t), &&\text{ for }t>0,\\
\bar{u}(0,x)=u(\eps,x)&&\text{ for }x\in (\alpha s(\eps), M),
\end{array}\right.
\]
where $g_\eps(t):=u(t+\eps,M)$.

Note that $u$ and $\bar{u}$ are smooth in their support, respectively, for $t>0$ by parabolic interior Schauder estimates (see \cite{bf}).
Then applying \cite[Theorem 3.1]{fasano1981new} to
\[
\int_{\alpha s(t+\eps)}^{x}\int_{\alpha s(t+\eps)}^y \left(1-\alpha u(t+\eps,z)\right)dz dy\quad \text{and}\quad \int_{\alpha \bar{s}(t)}^{x}\int_{\alpha \bar{s}(t)}^y \left(1-\alpha \bar{u}(t,z)\right)dz dy,
\]
yields $\bar{u}(t,\cdot)=u(t+\eps,\cdot)$ and $\bar{s}(t)=s(t+\eps)$ for all $t\geq 0$ small enough. 
By \eqref{aa1} and \eqref{55555}, $p(\frac\delta2,\cdot)=u(\frac\delta2,\cdot+\alpha s(\frac\delta2))$ satisfies \eqref{1.2}. Finally, Lemma \ref{T.1.1} yields the conclusion.
\end{proof}

\smallskip
{\bf Self-similar solutions}: For any $\beta>0$, consider the following functions
\beq\lb{2.13}
U(t,x;c,\beta):={2\alpha^{-1}\beta}e^{\beta^2}\int_\beta^{\frac{x-c}{\sqrt{2t}}}e^{-z^2}dz\quad \text{ and }\quad S(t;c,\beta)=\alpha^{-1}(c+\beta\sqrt{2t}).
\eeq
These functions come from self-similar solutions to the supercooled Stefan problem (see  \cite{And04, Car45}): the pair $(U(t,x;c,\beta),S(t;c,\beta))$ satisfies
\[
\left\{ \begin{array}{lcl}
U_t=\frac{1}{2}U_{xx} &&\text{ in }\{(t,x): x>\alpha S(t;c,\beta),\, 0<t<T\},\\
S'(\cdot;c,\beta)=\frac{1}{2}U_x(\cdot,\alpha S(\cdot;c,\beta);c,\beta) &&\text{ on }(0,\infty).
\end{array}\right.
\]

It is easy to see that for all $x,t>0$ ,
\beq\lb{2.3}
\beta_\infty(\beta):=\alpha U(0,x;c,\beta)=\lim_{x\to \infty }\alpha U(t,x;c,\beta)={2\beta}e^{\beta^2}\int_\beta^{\infty}e^{-z^2}dz,
\eeq
and for all $\beta>0$, $\beta_\infty(\beta)$ takes all values in $(0,1)$. Indeed, for all $\beta>0$,
\[
\beta_\infty(\beta) = 2\beta \int_0^\infty e^{-y^2-2\beta y}dy=\int_0^\infty e^{-(2\beta)^{-2}z^2-z}dz<\int_0^\infty e^{-z}dz=1.
\]
Moreover, we have the following estimate for all $\beta\geq 1$,
\beq\lb{2.12}
\begin{aligned}
\beta_\infty(\beta) =\int_0^\infty e^{-(2\beta)^{-2}z^2-z}dz\geq  \int_0^{\beta} e^{-z}(1-(2\beta)^{-2}z^2)dz\geq 1-\frac{1}{\beta^2},
\end{aligned}
\eeq
where we used that $\int_0^\infty e^{-z}z^2dz=2$.
From the equality in \eqref{2.12}, we also know that $\beta_\infty(\cdot)$ is an increasing function.

\subsection{Comparison lemmas}
\label{sc22}
We first present the following comparison lemma. 
Instead of comparing the solutions to  the super-cooled Stefan problem \eqref{2.7}, we consider a linear combination of the solution and its integration. This will allow us to compare a solution that is possibly large at some points with the self-similar solution $U$ (which is no greater than $\beta_\infty/\alpha$).

\begin{lemma}\lb{L.2.1}
Suppose $(u_1,s_1),(u_2,s_2)$ are two classical solutions to 
\eqref{2.7} in $[0,T)\times [0,\infty)$. Let $\gamma\in [0,1]$ and for $i=1,2$ write
\[
v_i(t,x):=\gamma u_i(t,x)+\left(\int_{\alpha s_i(t)}^x u_i(t,y)dy+s_i(t)\right).
\]
If the following holds for all $t\in [0,T)$,
\begin{itemize}[itemsep = 3 pt]
    \item[1. ] $s_1(t)\leq s_2(t)$,
    
    \item[2. ] $v_1(0,x)\geq v_2(0,x)$ for all $x>\alpha s_2(0)$,
    
    \item[3. ]  $\liminf_{x\to\infty}(v_1(t,x)-v_2(t,x))>0$ locally uniformly in $t$,
    
    \item[4. ] $v_1(t,\alpha s_2(t))\geq v_2(t,\alpha s_2(t))$,

\end{itemize}
then for all $t\in [0,T)$ and $x>\alpha s_2(t)$,
\[
v_1(t,x)\geq v_2(t,x).
\]
\end{lemma}

\begin{proof}
Direct computation yields
\[
\left\{ \begin{array}{lcl}
(v_i)_t=\frac{1}{2}(v_i)_{xx} &&\text{ in }\{(t,x): x>\alpha s_i(t),\, t\in (0,T)\},\\
v_i(t,\alpha s_i(t))=s_i(t),\, (v_i)_x(t,\alpha s_i(t))=2\gamma s_i'(t) &&\text{ for }t\in [0,T).
\end{array}\right.
\]

By the assumptions, we have $w:=v_1-v_2$ satisfies
\[
w_t=\frac{1}{2}w_{xx}\quad \text{ in }\{(t,x): x>\alpha s_2(t),\, t\in (0,T)\},
\]
and $w(t,\alpha s_2(t))\geq 0$. Also since the condition on the initial data yields $w(0,\cdot)\geq 0$ on $\{x>\alpha s_2(0)\}$, and due to the condition 3, the conclusion follows from the maximum principle (in bounded space-time domain $\{(t,x): x\in(\alpha s_2(t),N),\, t\in [0,T']\}$ for any $T'<T$ and for sufficiently large $N$) (see e.g., \cite[Corollary 6.26]{Lieberman}).
\end{proof}

We also use the following transformation as done in \cite{classical8}: recall that $(u,s)$ is a solution to \eqref{2.7}, and define
\beq\lb{2.10}
m(t,x):=\int_{\alpha s(t)}^{x}\int_{\alpha s(t)}^y \left(1-\alpha  u(t,z)\right)dz dy,
\eeq
and
\[
m_0(x):= \int_{\alpha s(0)}^{x}\int_{\alpha s(0)}^y \left(1-\alpha u(0,z)\right)dz dy.
\]
Then $m$ satisfies the following problem
\beq\lb{3.1}
\left\{
\begin{aligned}
    & m_t=\frac{1}{2}m_{xx}-1 &&\text{ in }\{(t,x): x>\alpha s(t),\, 0<t<T\},\\
    & m(t,\alpha s(t))=m_x(t,\alpha s(t))=0   &&\text{ for }t\in [0,T),\\
    &m(0,x)=m_0(x)&&\text{ for }x\in [\alpha s(0),\infty).
\end{aligned}\right.
\eeq

Note that in \eqref{3.1}, differentiation of $s(t)$ is not involved. However intuitively $s(t)$ can still be identified through the equation because of the two boundary information (if known $s(t)$, only one boundary data is needed to solve for $m$).
We have the following comparison principle.

\begin{lemma}\lb{T.3.1}
Suppose $(u_1,s_1),(u_2,s_2)$ are two classical solutions to \eqref{2.7} in $[0,T)\times [0,\infty)$ with non-negative initial data $u_{1,0}, u_{2,0}$ that are supported in $(\alpha s_1(0),\infty)$ and $(\alpha s_2(0),\infty)$ respectively. Write their corresponding transformations as $m_1,m_2$. Suppose the following holds for all $t\in [0,T)$
\begin{itemize}
    \item[1. ] $(s_2(0)\leq s_1(0),s_2'(0)< s_1'(0))\text{ or } (s_2(0)< s_1(0))$,
    
    \item[2. ] $m_{2}(0,\cdot)\geq  m_{1}(0,\cdot)$ in $[0,\infty)$,
    
    \item[3. ] $m_2(t,x)\geq 0$ for $x>\alpha s_2(0),$
    
    \item[4. ] $\liminf_{x\to\infty }(m_2(t,x)- m_1(t,x))>0,\,\liminf_{x\to\infty} m_2(t,x)> 0$.
\end{itemize}
Then $s_2(t)< s_1(t)$, and $m_{1}(t,\cdot)< m_{2}(t,\cdot)$ for all $t\in (0,T)$.
\end{lemma}

\begin{proof}
The proof is identical to that in Lemma 3.1 and Remark 3.2 in \cite{classical8}.
\end{proof}

Recall the self-similar solutions $(U,S)$ given in \eqref{2.13}. The idea of controlling $N(t)$ in \eqref{eq:MVHLinFP} and then proving long time existence of solution is to apply the above two comparison principles to compare the free boundaries of a general solution $u$ and the self-similar solution $U$ with certain choices of $c,\beta$. 
We first show that $S(t;c,\beta)\geq s(t)$ for $c>0$.

\begin{lemma}\lb{L.3.0}
Let $(u,s)$ be a classical solution to the super-cooled Stefan problem \eqref{2.7} for $t\in [0,T)$, such that
\beq\lb{c.3.4}
\|u_0\|_{L^1(\bbR^+)}\leq 1,\quad\|u_0\|_\infty<\infty, \quad  \lim_{x\to 0 } u_0(x)+\lim_{x\to\infty } u_0(x)=0,
\eeq
and 
\beq\lb{c.3.5}
\int_0^x (1-\alpha u_0(y))dy>0, \quad \mbox{for all } x\in (0,\infty).
\eeq
There exists $C_0\geq 2$ depending only on $u_0$ such that for all $\beta\geq C_0$, we have 
\[
s(t)\leq S(t;0,\beta)=\alpha^{-1}\beta\sqrt{2t}\quad\text{ for all $t\in [0,T)$.}
\]

\end{lemma}
\begin{proof}
Recall that $u(t,x)$ solves the supercooled Stefan problem: 
\begin{align*}
    & u_t=\frac{1}{2}u_{xx}\quad \text{ in }\{x\in (\alpha s(t),\infty),t\in [0,T)\},\\
    & u(t,\alpha s(t))=0,\quad\text{ and }\quad u_x(t,\alpha s(t))=s'(t).
\end{align*}
Since $s(\cdot)\geq 0$, comparison principle yields $u\leq \tilde{u}$ where the latter is given by \eqref{utilde}.
By the assumption $u(0,x)\to 0$ as $x\to\infty$, we have $\lim_{x\to\infty }\tilde{u}(t,x)=0$ (uniformly in $t$), which implies that $\lim_{x\to\infty }u(t,x)=0$ uniformly in $t$. Thus for all $t\in [0,T)$ we have 
\beq\lb{2.15}
\liminf_{x\to\infty} U(t,x;c,\beta)=\alpha^{-1}\beta_\infty>0= \lim_{x\to\infty} u(t,x).
\eeq

Let $m$ be defined as in \eqref{2.10}, and we also define 
\[
M(t,x;c,\beta):=\int_{\alpha S(t)}^{x}\int_{\alpha S(t)}^y \left(1-\alpha U(t,z;c,\beta)\right)dz dy.
\]
It follows from \eqref{2.15} that 
\beq\lb{2.6}
\liminf_{x\to\infty}(m(t,x)-M(t,x;c,\beta))> 0.
\eeq

Since the total mass of $u(t,\cdot)$ is bounded from above by $1$, we know
\beq\lb{2.16}
m_0(x)= \int_{0}^{x}\int_{0}^y \left(1-\alpha u_0(z)\right)dz dy\geq \frac{x^2}{2}-\alpha x.
\eeq
Also by the assumption \eqref{c.3.4}--\eqref{c.3.5}, we obtain $m_0(x)>0$ for all $x\in (0,\infty)$ and $m_0(x)\geq \frac{x^2}{4}$ for $x>0$ small enough.
In view of \eqref{2.3} and \eqref{2.12}, there exists $C_0\geq 2$ depending only on $u_0$ such that for all $c\geq 0$, if $\beta\geq C_0$, we have 
\[
M(0,\cdot;c,\beta)=  \frac{1-\beta_\infty}{2}(x-c)_+^2\leq  \frac{x^2}{2\beta^2}< m_0(\cdot)\quad\text{ in }(0,3\alpha].
\]
While for $x>3\alpha$, \eqref{2.16} yields
\[
M(0,\cdot;c,\beta)\leq  \frac{x^2}{8}< \frac{x^2}{2}-\alpha x\leq m_0(\cdot).
\]
Note that $S(0;c,\beta)=\alpha^{-1}c $. Then it follows from Lemma \ref{T.3.1} with $u_1=U,u_2=u$ that
$s(t)< S(t;c,\beta)$ for all $t\in [0,T)$ and $c>0$. By passing $c\to 0$, we get
\beq\lb{2.4}
s(t)\leq  S(t;c,\beta)\quad\text{ for all $t\in [0,T)$ and $c\geq 0$.}
\eeq
\end{proof}

In Lemma \ref{L.3.0}, we only applied the second comparison lemma (Lemma \ref{T.3.1}). From its proof, note that we can replace the assumption \eqref{c.3.5} by $\int_0^x\int_0^y (1-\alpha u_0(z))dzdy>0$ for all $x>0$, which is slightly weaker than \eqref{c.3.5}.

In order to compare $s(t),S(t;c,\beta)$ for $c<0$, we also need the following computations.
For some $\gamma\in (0,1)$ to be determined, write 
\beq\lb{d.V}
V(t,x;c,\beta):=\gamma U(t,x;c,\beta)+\left(\int_{\alpha S(t;c,\beta)}^x U(t,y;c,\beta)dy+S(t;c,\beta)
\right)
\eeq
and
\beq\lb{d.v}
v(t,x):=\gamma u(t,x)+\left(\int_{\alpha s(t)}^x u(t,y)dy+s(t)
\right).
\eeq

\begin{lemma}\lb{L.3.1}
Under the assumptions of Lemma \ref{L.3.0}, there exists $\gamma_0\in (0,1)$ (depending only on $u_0$) such that for all $\gamma\in (0,\gamma_0)$, and $\beta$ satisfying
\[
\beta\geq \max\{C_0,100,2\sqrt{\alpha/\gamma},4\sqrt{2T}\gamma^{-1},4\sqrt{2T}(\alpha\gamma)^{-1},4\sqrt{2T}\gamma^{-1}(\ln (4\sqrt{2T}\gamma^{-1}))^2\},
\] 
and $c\in (-\alpha^{-1}\beta\sqrt{2T},0)$, if $s(t)\geq S(t;c,\beta)$ for all $t\in [0,t_1]$ for some $t_1<T$, then the following inequalities hold
\begin{align}
    &\liminf_{x\to\infty}(V(t,x;c,\beta)-v(t,x))=\infty&&\text{ for all }t\in [0,T),\lb{i.3.1}\\
&V(0,\cdot;c,\beta)\geq v(0,\cdot)&&\text{ in }[0,\infty),\lb{i.3.2}\\
&V(t,x;c,\beta)\geq \alpha^{-1}x &&\text{ for all $t\in [0,t_1]$ and $x\leq \beta\sqrt{2T}$}.\lb{i.3.3}
\end{align}
\end{lemma}

\begin{proof}
Due to \eqref{2.15}, clearly \eqref{i.3.1} holds for all $t,c,\beta$.
Using \eqref{2.3}, \eqref{2.12}, and $c\in ( -\alpha^{-1}\beta\sqrt{2T},0)$ yields for all $x\geq 0$, %
\[
V(0,x;c,\beta)\geq {\gamma{\alpha}^{-1}\beta_\infty}+{\alpha}^{-1}\beta_\infty x-{(\alpha\beta)}^{-1}\sqrt{2T}.
\]
Let us assume $\beta\geq \max\{4\gamma^{-1}\sqrt{2T},2\}$ and then $\beta_\infty\geq \frac{3}{4}$. 
By \eqref{2.12} and \eqref{d.V}, we get
\beq\lb{2.9}
\begin{aligned}
    &\quad\, \, V(0,x;c,\beta)-v(0,x)\\
    &\geq \gamma (\frac{3}{4}\alpha^{-1}-u(0,x))-\gamma(4\alpha)^{-1}-\alpha^{-1}(1-\beta_\infty)x+\alpha^{-1}\int_0^x(1-\alpha u_0(y))dy\\
    &\geq \gamma((2\alpha)^{-1}-u(0,x))+\alpha^{-1}\left(-\beta^{-2}x+\int_0^x(1-\alpha u_0(y))dy\right).
\end{aligned}
\eeq

In view of \eqref{c.3.4}, there exists $c,A>0$ such that
\[
u(0,x)\leq (4\alpha)^{-1}\text{ for all }x\in [0,c],\quad\text{ and }\quad u(0,x)\leq A\text{ for all }x\in [0,\infty).
\]
When $x\geq 2\alpha$, the right-hand side of \eqref{2.9}
\[
\geq -\gamma A+\alpha^{-1}(-\beta^{-2}x+x-\alpha)\geq -\gamma A+(1-2\beta^{-2})\geq 0,
\]
if $\beta\geq 2$ and $\gamma\leq (2A)^{-1}$.
Next when $x\leq  c$ (then $u(0,x)\leq (4\alpha)^{-1}$), \eqref{2.9} and $\beta\geq 2$ yield again $
V(0,x;c,\beta)-v(0,x)\geq 0$.
Lastly by the assumption \eqref{c.3.5}, there exists $\eps>0$ such that $\int_0^x(1-\alpha u_0(y))dy\geq \eps$ for $x\in [c,2\alpha]$. Thus we get
\begin{align*}
    V(0,x;c,\beta)-v(0,x)\geq \gamma((2\alpha)^{-1}-A)+(\alpha^{-1}\eps-2\beta^{-2})\geq 0
\end{align*}
if $\beta\geq 2\sqrt{\alpha/\eps}$ and $\gamma=\gamma(\eps,A)\leq \eps$ is small enough. Overall, we find that there exists $\gamma$ depending only on $u_0$ such that
\eqref{i.3.2}
holds for all $\beta\geq \max\{4\gamma^{-1}\sqrt{2T},2\sqrt{\alpha/\gamma},2\}$.

To prove the last inequality \eqref{i.3.3}, we need a lower bound on $U$. Below we write $S(t):=S(t;c,\beta)$ for abbreviation of notation. It follows from \eqref{2.13}, \eqref{2.12} and the fact $\int_0^\infty e^{-z}z^2dz=2$ that for $\beta\geq 2$, if $ \frac{x-\alpha S(t)}{\sqrt{2t}}\geq\frac{2\ln \beta}{\beta}$, we have
\beq\lb{1111}
\begin{aligned}
\alpha U(t,x;c,\beta)&= {2\beta}e^{\beta^2}\int_\beta^{\frac{x-c}{\sqrt{2t}}}e^{-(z')^2}dz'= \int_0^{2\beta\left(\frac{x-\alpha S(t)}{\sqrt{2t}}\right)} e^{-z-(2\beta)^{-2} z^2}dz\\
&\geq  \int_0^{2\beta\left(\frac{x-\alpha S(t)}{\sqrt{2t}}\right)} e^{-z}(1-(2\beta)^{-2}z^2)dz\geq \int_0^{2\beta\left(\frac{x-\alpha S(t)}{\sqrt{2t}}\right)} e^{-z}dz-\frac{1}{2\beta^2}\\
&= 1-\exp\left(-2\beta\left(\frac{x-\alpha S(t)}{\sqrt{2t}}\right)\right)-\frac{1}{2\beta^2}\geq 1-\frac{1}{\beta^2}.
\end{aligned}
\eeq
When $\beta\geq 2$ and $ \frac{x-\alpha S(t)}{\sqrt{2t}}<\frac{2\ln \beta}{\beta}$, since $\frac18\geq \frac{\ln\beta}{2\beta^3}$, direct computation yields
\beq\lb{2222}
\begin{aligned}
\alpha U(t,x;c,\beta)&=  \int_0^{2\beta\left(\frac{x-\alpha S(t)}{\sqrt{2t}}\right)} e^{-z-(2\beta)^{-2} z^2}dz\geq  \int_0^{2\beta\left(\frac{x-\alpha S(t)}{\sqrt{2t}}\right)} e^{-z-(2^{-1}\beta^{-3}\ln\beta) z}dz\\
&\geq \frac{8}{9}\left(1-\exp \left(-\frac{9\beta}{4}\left(\frac{x-\alpha S(t)}{\sqrt{2t}}\right)\right)\right).\\
&\geq \frac{1}{2}\min\left\{\beta \left(\frac{x-\alpha S(t)}{\sqrt{2t}}\right),1\right\}.
\end{aligned}
\eeq
Using these estimates, for any $t\in [0,T)$ and
\[
x\in \left[ \alpha S(t), \frac{2\sqrt{2t}\ln \beta}{\beta}+\alpha S(t) \right),
\]
we obtain 
\begin{align*}
V(t,x;c,\beta)\geq \frac{\gamma}{2\alpha}  \min\left\{\beta\left(\frac{x-\alpha S(t)}{\sqrt{2t}}\right),1\right\}+ S(t)
\end{align*}
and so to have $V(t,x;c,\beta)\geq {\alpha}^{-1}x$, it suffices to require $\beta\geq 2\gamma^{-1}\sqrt{2T}$ and $\frac{\beta}{\ln \beta}\geq 4\gamma^{-1}\sqrt{2T}$ which is indeed guaranteed by the assumption on $\beta$.
Next for
\beq\lb{case2}
x\in \left[ \frac{2\sqrt{2t}\ln \beta}{\beta}+\alpha S(t), \beta\sqrt{2T} \right],
\eeq
by \eqref{1111} and \eqref{2222}
we find (writing $U(t,y):=U(t,y;c,\beta)$)
\begin{align*}
&\quad\,\, \alpha\int_{\alpha S(t)}^x U(t,y;c,\beta )dy+\alpha S(t)\\
&= \alpha \int_{\alpha S(t)}^{\alpha S(t)+\frac{\sqrt{2t}}{\beta}} U(t,y )dy+\alpha \int_{\alpha S(t)+\frac{\sqrt{2t}}{\beta}}^{\alpha S(t)+\frac{2\sqrt{2t}\ln\beta}{\beta}} U(t,y )dy+\alpha \int_{\alpha S(t)+\frac{2\sqrt{2t}\ln\beta}{\beta}}^{x} U(t,y )dy+\alpha S(t)\\
&\geq 
\sqrt{2t}\left(\frac{1}{4\beta}+\frac{1}{2}\left(\frac{2\ln\beta}{\beta}-\frac{1}{\beta}\right)+\left(1-\frac{1}{\beta^2}\right)\left(\frac{x-\alpha S(t)}{\sqrt{2t}}-\frac{2\ln\beta}{\beta}\right)\right)+\alpha S(t)\\
&\geq -\frac{\sqrt{2t}}{4\beta}-\frac{\sqrt{2t}\ln\beta}{\beta}+\left(1-\frac{1}{\beta^2}\right)x+\frac{\alpha S(t)}{\beta^2}.
\end{align*}
Due to $\alpha S(t)\geq c\geq -\alpha^{-1}\beta\sqrt{2T}$, then for $x$ satisfying \eqref{case2} we obtain
\begin{align*}
\alpha V(t,x;c,\beta)-x &\geq \gamma\left(1-\frac1{\beta^2}\right)-\frac{\sqrt{2t}}{4\beta}-\frac{\sqrt{2t}\ln\beta}\beta-\frac{x}{\beta^2}+\frac{\alpha S(t)}{\beta^2}\\
&
\geq\frac{3}{4}\gamma-\sqrt{2T}\left(\frac{5}{4\beta}+\frac{\ln\beta}\beta+\frac{1}{\alpha\beta}\right)\geq 0,
\end{align*}
whenever 
\[
\beta\geq \max\{100,4\sqrt{2T}(\alpha\gamma)^{-1},4\sqrt{2T}\gamma^{-1}(\ln (4\sqrt{2T}\gamma^{-1}))^2\}.
\]
This is because for $c_0:=4\sqrt{2T}\gamma^{-1}>0$, either $\frac{\beta}{\ln\beta}\geq \frac{ 100}{\ln 100}\geq c_0$ or $\frac{\beta}{\ln\beta}\geq\frac{c_0(\ln c_0)^2}{\ln(c_0(\ln c_0)^2)}\geq c_0$.
We proved \eqref{i.3.3}.
\end{proof}

In the following proposition we show that if the curves $x=\alpha s(t)$ and $x=\alpha S(t;c,\beta)$ with $c<0$ intersect at time $t=t_0>0$, then they can only insect at $t=t_0$ for all $t\in [0,T)$. 
\begin{lemma}\lb{P.1}
Under the assumptions of Lemma \ref{L.3.0}, let $(\gamma,\beta)$ be from Lemma \ref{L.3.1}. 
For any fixed $t_0\in (0,T)$, if there is a value $c\in ( -\alpha^{-1}\beta\sqrt{2T},0)$ such that $s(t_0)=S(t_0;c,\beta)$, then for all $t\in (0,T)$,
\beq\lb{2.5}
s(t)-S(t;c,\beta)\text{ changes sign from positive to negative at }t=t_0
\eeq
(i.e. $s(t)-S(t;c,\beta)>0$ for all $t<t_0$ and $s(t)-S(t;c,\beta)< 0$ for all $t>t_0$).
\end{lemma}

\begin{proof}
It follows from $c<0$ and $S(0;c,\beta)=\alpha^{-1}c<0$ that $S(t;c,\beta)<s(t)$ for $t$ sufficiently small. Suppose for contradiction that there is $t_1<t_0$ such that $S(t_1;c,\beta)=s(t_1)$ and $S(t;c,\beta)<s(t)$ for $t<t_1$. 

\quad Lemma \ref{L.3.0} yields $s(t)\leq \alpha^{-1}\beta\sqrt{2T}$ for all $t\in [0,T)$. Hence it follows from \eqref{i.3.3} that 
\[
V(t,\alpha s(t);c,\beta)\geq s(t)=v(t,\alpha s(t))\quad\text{ for } t\in [0,T).
\]
Then, using the assumption that $S(t;c,\beta)<s(t)$ for $t<t_1$, Lemma \ref{L.3.1} and Lemma \ref{L.2.1} (with $v_1=V,v_2=v$ where $V,v$ are given in \eqref{d.V}, \eqref{d.v} respectively) yield that
\beq\lb{2.8}
\begin{aligned}
\gamma U(t,x;c,\beta)+\int_{\alpha S(t)}^x U(t,y;c,\beta)dy&+S(t;c,\beta)=V(t,x;c,\beta)\\
&\geq v(t,x)=\gamma u(t,x)+\int_{\alpha s(t)}^x u(t,y)dy+s(t)
\end{aligned}
\eeq
for all $(t,x)\in\{x>\alpha s(t),t\in [0,t_1]\}$.
By the strong maximum principle (or Hopf's Lemma), we have $S'(t_1;c,\beta)>s'(t_1)$. 

Next consider 
\[
Z(x):=\int_{\alpha s(t_1)}^x\int_{\alpha s(t_1)}^y (U(t_1,z;c,\beta)-u(t_1,z))dz
\]
which, by \eqref{2.8} and the assumption that $S(t_1;c,\beta)=s(t_1)$, satisfies 
\[
\gamma Z''(x)+Z'(x)\geq 0\quad\text{ and }\quad Z(\alpha s(t_1))=Z'(\alpha s(t_1))=0.
\]
This implies that $Z(x)\geq 0$ for all $x>\alpha s(t_1)$. Therefore the definitions of $m,M$ yield
$m(t_1,\cdot)\geq M(t_1,\cdot;c,\beta)$. Clearly for all $t\in [0,T)$, by maximum principle and \eqref{2.6}, we have that $m(t,\cdot)>0$ for $x>\alpha s(t)$. In view of Lemma \ref{T.3.1} again, we obtain $S(t;c,\beta)>s(t)$ for all $t>t_1$ which contradicts with the assumption that $S(t_0;c,\beta)=s(t_0)$ with $t_0>t_1$. Hence $S(t;c,\beta)<s(t)$ for all $t<t_0$.
By going over the arguments in the above again, we also find that $S(t;c,\beta)>s(t)$ for all $t>t_0$.
\end{proof}

\subsection{Proofs of Theorem \ref{T.4.1} and Corollary \ref{T.3.3}}
\label{sc23}
We start by proving Theorem \ref{T.4.1}.
\begin{proof}[Proof of Theorem \ref{T.4.1}]
Suppose $(p,N)$ is a weak solution to \eqref{eq:MVHLinFP} in $ [0,T)$. Let $s(t)=\int_0^tN(r)dr$ for all $t\in [0,T)$, and $u,m$ be defined as in \eqref{2.2}, \eqref{2.10} respectively. Then $u_0=p_0$. Since $u$ solves the heat equation, $u(t,\cdot)$ is bounded in $L^\infty$ norm for any $t\in (0,T)$ (see e.g., \cite[Theorem 6.17]{Lieberman}). By the argument before \eqref{2.15}, we know that $u(t,x)\to 0$ as $x\to\infty$ locally uniformly in $t$ and so the same property holds for $p(t,x)$. Therefore there exists $C>0,\eps\in (0,T)$ such that
\[
\int_0^x (1-\alpha p(t,y))dy>0 \quad\text{ for all }x\geq C \text{ and }t\leq \eps.
\]
Due to \eqref{c.3.3} and $u\leq \tilde{u}$ with $\tilde{u}$ from Lemma \ref{l.clascsol}, we can assume without loss of generality that for the same $C,\eps>0$, $p(t,x)< \frac1\alpha$ for all $t\leq \eps$ and $x\leq \frac{1}{C}$. By further taking $\eps$ to be small enough, the weak formulation of solutions and the assumption on $p_0$ imply that
$\int_0^x (1-\alpha p(t,y))dy>0$ for $x\in [\frac{1}{C},C]$ and $t\leq \eps$. Then the assumptions \eqref{c.3.4}--\eqref{c.3.5} hold with $u_0$ replaced by $u(\eps,\cdot)$.  Hence, by starting at a small time $t=\eps$ instead of $t=0$, we can assume without loss of generality that $p_0$ is uniformly bounded in $L^\infty$, and the solution $(p,N)$ is a classical solution.

We take $\beta>0$ to be the smallest  constant satisfying the condition in Lemma \ref{L.3.1}. For any fixed $t_0\in (0,T)$, \eqref{2.4} implies $s(t_0)\leq S(t_0;0,\beta)$. Firstly if $s(t_0)=S(t_0;0,\beta)$, we claim that $s(t)=S(t;0,\beta)$ for all $t\leq t_0$. If this is not true, then there exist $t_1<t_0$ and $c<0$ sufficiently close to $0$ such that $s(t_1)<S(t_1;c,\beta)$. According to Lemma \ref{P.1}, we must have $s(t)<S(t;c,\beta)$ for all $t>t_1$ which is a contradiction because then $s(t_0)<S(t_0;c,\beta)<S(t_0;0,\beta)$. So in this case we obtain $s(t)=S(t;0,\beta)$ for all $t\leq t_0$. Also, in the case, for all $t\in (0,t_0)$ we have
\[
N(t)=s'(t)=S'(t;0,\beta)=\alpha^{-1}\beta(2t)^{-\frac{1}{2}}.
\]

Next we consider the case when $s(t_0)<S(t_0;0,\beta)$, which by definition is the same as $s(t_0)< \alpha^{-1}\beta\sqrt{2t_0}$. Thus Lemma \ref{P.1} yields that the curve $x=S(t;s(t_0)-\alpha^{-1}\beta\sqrt{2t_0},\beta)$ intersects with $x=s(t)$ at exactly one point $t=t_0$ for all $t\in [0,T)$ (notice here in terms of $S(t;c,\beta)$, $c$ takes the value of $s(t_0)-\alpha^{-1}\beta\sqrt{2t_0}\geq -\alpha^{-1}\beta\sqrt{2T}$, and so the assumption of Lemma \ref{P.1} is satisfied). Therefore
\[
N(t)=s'(t)\leq S'(t;s(t)-\alpha^{-1}\beta\sqrt{2t},\beta)=\alpha^{-1}\beta(2t)^{-\frac{1}{2}}.
\]

From the choice of $\beta$, and by varying $T$ (to be $t$) in the above arguments, we obtain for all $t\in [0,T)$,
\begin{align*}
N(t)\leq \alpha^{-1}\beta(2t)^{-\frac{1}{2}}\leq C\alpha^{-1}(1+\alpha^{-1}+(1+\alpha^\frac12)t^{-\frac{1}{2}}+(\ln t)^2) 
\end{align*}
for some $C$ depending only on $p_0$.
We can now conclude the proof by Lemma \ref{l.clascsol}.
\end{proof}

Now we proceed to proving Corollary \ref{T.3.3}.
\begin{proof}[Proof of Corollary \ref{T.3.3}]
To prove the first statement, in view of Theorem \ref{T.4.1}, it suffices to have
\beq\lb{4.2}
\int_0^x\left(1-\alpha \delta_{x_0}(y)\right)dy> 0
\eeq
for all $x>0$. Direct computation yields
that this is equivalent to 
${x}> \alpha 1_{x>x_0}$, which is the same as $\alpha< x_0$.

For the second statement, it follows from Proposition \ref{T.1.2} that if
\[
\alpha>2\int_0^\infty x \delta_{x_0}dx=2x_0,
\]
then any weak solution cannot exist for all time. 
In view of Lemma \ref{l.clascsol}, $(p_{x_0})_x\to \infty$ in finite time.
\end{proof}

\section{Solutions to the Fokker-Planck equation \eqref{eq:MVHLogFP}}
\label{sc3}

In this section, we consider the Fokker-Planck equation \eqref{eq:MVHLogFP}.
First of all, comparing to the equation \eqref{eq:MVHLinFP}, the equation is non-local. Since the total mass is decreasing after assuming $\lim_{x\to\infty}q(t,x)=0$ for all $t$ (there is no mass coming from $x=\infty$), the non-local form of $\lambda(t)$ yields a fast growth of it as $t$ increases. Hence it is more likely that $\lambda(t)$ grows to infinity in finite time comparing to $N(t)$ in \eqref{eq:MVHLinFP}.  

Recall the following result from \cite[Proposition 4.1]{NS19}, which proves existence and uniqueness of a generalized solution up to the first blow-up.

\begin{lemma}\lb{T.5.1}{\rm (\cite{NS19})}
Let $q_0(\cdot)\in W_2^1([0,\infty))$ with $q_0(0)=0$. For any $T>0$, there exist a time $t_{\reg}\in (0,T]$ and a function $\lambda\in L^2_{\loc}([0,t_{\reg}))$ such that for all $T'\in (0,t_{\reg})$ the unique solution to the Fokker-Planck equation \eqref{eq:MVHLogFP} in $W_2^{1,2}([0,T']\times [0,\infty))$ satisfies
\[
\lambda(t)=-\frac{1}{2}\frac{q_x(t,0)}{\int_0^\infty q(t,y)dy}\quad \text{ for almost every }t\in [0,T']
\]
and $\lim_{T'\uparrow t_{\reg}}\|\lambda\|_{L^2[0,T']}=\infty$ if $t_{\reg}<T$.
\end{lemma}

The idea is to prove $t_{\reg} = \infty$ for suitable $q_0$, $\beta$ and $\alpha$.
Let us perform a transformation which turn the non-local equation \eqref{eq:MVHLogFP} into a local one. 
Denoting $\bar{q}(t):=\int_0^\infty q(t,x)dx$, we then have
$\bar{q}'(t)=\lambda(t)\bar{q}(t)$, and $r:={q}/{\bar{q}}$ satisfies
\beq\lb{5.2}
\left\{ \begin{array}{lcl}
r_t=\frac{1}{2}r_{xx}-(\alpha \lambda(t)+\beta)r_x -\lambda(t)r &&\text{ in }[0,\infty)^2,\\
\lambda(t)=-\frac{1}{2}{r_x(t,0)},\quad r(t,0)=0 &&\text{ on }[0,\infty),\\
r(0,x)={q_0(x)}/{\bar{q}(0,\cdot)} && \text{ on }[0,\infty).
\end{array}\right.
\eeq

Clearly the equation preserves mass. 
Now let us prove Proposition \ref{T.1.2} (2).
Call $r_0(x): = {q_0(x)}/{\bar{q}(0,\cdot)}$.
\begin{proof}[Proof of Proposition \ref{T.1.2} (2)]
Let $(r,\lambda)$ be from \eqref{5.2} and suppose the solution exists for $t\in [0,T']$. 
Since $q\in W_2^{1,2}([0,T']\times \bbR)$, then $r\in W_2^{1,2}([0,T']\times \bbR)$ and $\int_0^\infty e^{-\mu x}r(t,x)dx$ is H\"{o}lder continuous in time by Morrey's inequality.
It follows from the equation that
\beq\label{2.2'}
\begin{aligned}
    \int_0^\infty e^{-\mu x}r(t,x)dx&=\int_0^\infty e^{-\mu x} r_0(x)dx+\int_0^t\int_0^\infty ( 2^{-1}e^{-\mu x}  r_{xx}(\tau,x)\\
    &\qquad\qquad -(\alpha\lambda(\tau)+\beta) e^{-\mu x} r_x(\tau,x)-\lambda(\tau) e^{-\mu x}r(\tau,x))dxd\tau\\
    &=\int_0^\infty e^{-\mu x} r_0(x)dx+(2^{-1}\mu^2-\mu\beta)\int_0^t\int_0^\infty e^{-\mu x}r(\tau,x)dxd\tau\\
    &\qquad\qquad-\int_0^t\lambda (\tau)((1+\alpha\mu)\int_0^\infty e^{-\mu x}r(\tau,x)dx-1)d\tau.
\end{aligned}
\eeq
Since $\mu> 2\beta$, writing $M_\mu(t):=(1+\alpha\mu)\int_0^\infty e^{-\mu x}r(t,x)dx-1$, we get
\[
M_\mu(t)\geq M_\mu(0) -(1+\alpha\mu)\int_0^t\lambda(\tau)M_\mu(\tau)  d\tau. 
\]
It then follows from Gronwall's inequality and $M_\mu(0)\geq 0$ by \eqref{4.3} that for all $t\geq 0$,
\[
(1+\alpha\mu)\int_0^\infty e^{-\mu x}r(t,x)dx\geq 1.
\]
Hence \eqref{2.2'} yields
\[
 \int_0^\infty e^{-\mu x}r(t,x)dx\geq \int_0^\infty e^{-\mu x}r_0(x)dx+ \frac{1}{2}\mu(\mu-2\beta)\int_0^t\int_0^\infty e^{-\mu x}r(\tau,x)dxd\tau,
\]
which, by Gronwall's inequality, implies 
\[
\int_0^\infty e^{-\mu x}r(t,x)dx\geq e^{\frac{\mu(\mu-2\beta)}{2} t}\int_0^\infty e^{-\mu x}r_0(x)dx.
\]
However since the total mass of $r(t,\cdot)$ is always $1$. The solution cannot exist for $t\geq T$ where $T$ is such that
\[
e^{\frac{\mu(\mu-2\beta)}{2} T}\int_0^\infty e^{-\mu x}r_0(x)dx=1.
\]
Then Lemma \ref{T.5.1} yields that the $L^2$ norm of $\lambda(t)$ blows up at the time when the solution fails to exist.
\end{proof}

To prove Theorem \ref{T.3.2}, we show an $L^2$ bound on $\lambda(t)$ for general initial data with small $\alpha$ and large $\beta>0$.
We need the following technical lemma which is similar to \cite[Lemma 3.4]{CC15}.
\begin{lemma}\lb{L.3.2}
Define $\phi:[0,\infty)\to \bbR$ by $\phi(x)=e^{-\frac{1}{1-x^2}}$ if $x\in [0,1)$, and $\phi(x):=0$ otherwise. Then the following properties hold:
\begin{enumerate}
    \item $0\leq -\phi_x\leq \phi$ on $(0,\frac{1}{4})$.
    \item There exists $C>0$ such that $\phi_x^2+\phi_{xx}^2+\phi_{xxx}^2\leq C\phi$ on $(0,1)$.
    \item There exists $x_0\in (0,1)$ such that $\phi_{xx}(x_0)=0$ and $\phi_{xx}\leq 0$ on $(0,x_0).$
\end{enumerate}
\end{lemma}
\begin{proof}
These are results of trivial (but tedious) computations. For the first property, 
direct computation yields for $x\in (0,\frac{1}{4})$,
$
-\phi_x=\frac{2x}{(1-x^2)^2}\phi\leq \phi.
$
For the remaining claims, it follows line by line from the proof of \cite[Lemma 3.4]{CC15}.
\end{proof}

\begin{proof}[Proof of Theorem \ref{T.3.2}]
For any fixed $\kappa\in (0,\frac{1}{8}]$, define
$
\omega(x):=2\kappa\, x e^{-\sqrt{2\kappa}x}$
and then it is not hard to check that $\omega_x(0)=2\kappa$ and
\beq\label{4.7}
\frac{1}{2}\omega_{xx}+\sqrt{2\kappa}\,\omega_x +\kappa \omega=0.
\eeq
Since $\kappa\leq \frac{1}{8} $, $\omega(\cdot)$ is an increasing function in $(0,\frac{1}{\sqrt{2\kappa}})\supseteq (0,1)$. 
Note that $\omega(\cdot)$ can be viewed as a stationary solution to \eqref{5.2} with $\lambda(t)\equiv-\kappa$ and any $\alpha,\beta$ satisfying $\alpha\kappa-\beta=\sqrt{2\kappa}$.

Recall $r=\frac{q}{\bar{q}}$ and  
we call
\beq\lb{4.4}
h(t,x):=\frac{r(t,x)}{\omega(x)}\quad\text{ and }\quad {j}(t):=-\frac{\lambda(t)}{\kappa}.
\eeq
By L'Hopital's rule, we have $h(t,0)={j}(t)$.
Then let $\phi(x)$ be from Lemma \ref{L.3.2} and recall that it is supported on $(0,1)$. Below we will sometimes drop $t,x$ from the notations of $h(t,x),r(t,x),\omega(x),j(t),\lambda(t),\phi(x)$.
It follows from equation \eqref{5.2} that
\begin{align*}
\frac{d}{dt}\int_0^\infty  h^2\omega\phi\, dx&=\int_0^\infty h(r_{xx}-2(\alpha \lambda+\beta) r_x-2\lambda r)\phi\, dx\\
&=\int_0^\infty (-h_x r_x-2(\alpha \lambda+\beta) r_x -2\lambda r)\phi\, dx-\int_0^\infty hr_x\phi_x dx+2\lambda {j} \phi(0).
\end{align*}
Using that $r_x=h_x \omega +h\omega_x$ and $\phi(0)=e^{-1}$, we get the above
\begin{align*}
    &=-\int_0^\infty (  h_x^2 \omega+h h_x \omega_x )\phi\, dx-2(\alpha \lambda+\beta)\int_0^\infty h( h_x \omega+h\omega_x)\phi\,dx\\
    &\qquad\qquad-2\lambda\int_0^\infty h^2\omega\phi\,dx-\int_0^\infty hr_x\phi_x dx-{2}e^{-1}{j}^2\kappa\\
    &=-\int_0^\infty h_x^2 \omega\phi\, dx+\frac{1}{2}\int_0^\infty h^2\omega_{xx}\phi\,  dx+\frac{1}{2}\int_0^\infty h^2 \omega_x\phi_x dx- (\alpha \lambda+\beta)\int_0^\infty h^2\omega_x\phi\,dx\\
    &\qquad\qquad-2\lambda\int_0^\infty h^2\omega\phi\,dx+(\alpha \lambda+\beta)\int_0^\infty h^2\omega\phi_xdx-\int_0^\infty h(h_x\omega +h\omega_x)\phi_x dx-e^{-1}j^2\kappa\\
    &=-\int_0^\infty h_x^2 \omega\phi\, dx+\frac{1}{2}\int_0^\infty h^2\omega_{xx} \phi\, dx+\frac{1}{2}\int_0^\infty h^2 \omega\phi_{xx} dx- (\alpha \lambda+\beta)\int_0^\infty h^2\omega_x\phi\,dx\\
    &\qquad\qquad-2\lambda\int_0^\infty h^2\omega\phi\,dx+(\alpha \lambda+\beta)\int_0^\infty h^2\omega\phi_xdx-\frac{\lambda^2}{e\kappa}.
\end{align*}
Using equation \eqref{4.7}, we obtain
\beq\lb{4.5}
\begin{aligned}
\frac{d}{dt}\int_0^\infty h^2 \omega\phi\, dx&=-\int_0^\infty  h_x^2\omega\phi\, dx-(\alpha \lambda+\beta+\sqrt{2\kappa})\int_0^\infty h^2\omega_x\phi\,dx\\
-(2\lambda +\kappa)\int_0^\infty & h^2\omega\phi\, dx+(\alpha \lambda+\beta)\int_0^\infty h^2\omega\phi_xdx+\frac{1}{2}\int_0^\infty h^2 \omega\phi_{xx} dx-\frac{\lambda^2}{e\kappa }.
\end{aligned}
\eeq



Let us write $I(t):=\int_0^\infty h^2 \omega\phi\, dx$, and $J(t):=\int_0^\infty h_x^2 \omega\phi\,  dx$. By Young's inequality, we get
\begin{align*}
   -\alpha \lambda\int_0^\infty h^2\omega_x\phi\,dx\leq -2\alpha  \lambda\int_0^\infty |hh_x \omega |\phi\, dx \leq \frac{1}{2}J(t)+{2\alpha^2}{\lambda^2}I(t).
\end{align*}
Since $\kappa\leq \frac18$ and $\phi$ is supported in $(0,1)$, it is easy to check that $\omega\leq 2\omega_x$. We get
\[
-(\alpha \lambda+\beta+\sqrt{2\kappa})\int_0^\infty h^2\omega_x\phi\,dx\leq \frac{1}{2}J(t)+2\alpha^2 \lambda^2I(t)-\frac{\beta}{2}
I(t).
\]
The third term on the right-hand side of \eqref{4.5} satisfies
\[
-(2\lambda +\kappa)\int_0^\infty  h^2\omega\phi\, dx\leq -2\lambda\int_0^\infty  h^2\omega\phi\, dx\leq\frac{\lambda^2}{2e\kappa}+2e\kappa I(t)^2.
\]

Now we consider the fourth and fifth terms in \eqref{4.5}. We claim that: there exists a constant $C$ (independent of $\alpha,\beta$) such that for any $\eps>0$ we have
\beq\lb{4.10a}
(\alpha \lambda+\beta)\int_0^\infty h^2\omega\phi_xdx\leq C\eps(I(t)+J(t))+\frac{C\alpha^2 \lambda^2}{\eps}I(t)+\frac{C\alpha^2 \lambda^2}{\eps},
\eeq
and
\beq\lb{4.10b}
\frac{1}{2}\int_0^\infty h^2 \omega\phi_{xx} dx\leq C\eps(I(t)+J(t))+\frac{C}{\eps}.
\eeq

Assume the claim holds, and then it follows from \eqref{4.5} and the above estimates that there exists $C\geq 1$ independent of $\alpha,\beta$ such that for any $\eps\in (0,1)$,
\begin{align*}
I'(t)&\leq -\frac{1}{2}J(t)-\kappa I(t)+C\eps(I(t)+J(t))+\frac{C}{\eps}\alpha^2\lambda^2 (I(t)+1)+\frac{C}{\eps}-\frac{\lambda^2}{2e\kappa}\\
&\qquad + 2e\kappa I(t)^2-\frac{\beta}{2}
I(t),
\end{align*}
which, by taking $\eps:=\min\{\frac{1}{2C},\frac{\kappa}{ C}\}$, implies that for $C_1:=\frac{C}{\eps}\geq 1$ (which is independent of $\alpha,\beta$, and $t$), we have
\beq\lb{3.88}
I'(t)\leq \lambda^2\left(-\frac{1}{2e\kappa}+C_1\alpha^2(I(t)+1)\right)+C_1-\frac{\beta}{4} I(t)+ I(t)\left(2e\kappa I(t)-\frac\beta4\right).
\eeq
Set $M:=\max\{1,I(0)\}$ and pick 
\[
C_0:=\max\left\{4C_1,8e\kappa M,\sqrt{4e\kappa C_1(M+1)}\right\}\geq 1,
\]
and so $\beta\geq C_0$ and $\alpha\leq 1/C_0$ imply that
\[
C_1-\frac{\beta M}4\leq 0,\quad 2e\kappa M-\frac\beta4\leq 0 \quad\text{ and }\quad -\frac{1}{2e\kappa}+C_1\alpha^2( M+1)\leq -\frac{1}{4e\kappa}.
\]
Then it is easy to see from the differential inequality \eqref{3.88} that $I(t)$ cannot exceed $M$.
Furthermore, with these choices of $\alpha$ and $\beta$, the differential inequality \eqref{3.88} yields $I'(t)\leq -\frac{1}{4e\kappa}\lambda^2+C_1$. Hence it follows the global $L^2$ bound of $\lambda$:
\[
\int_0^t \lambda^2(s)ds\leq {4e\kappa}(C_1t+M).
\]
for all $t>0$ such that the solution is well-defined.
We conclude by applying Lemma \ref{T.5.1}.

Let us now prove the estimates \eqref{4.10a}--\eqref{4.10b}. The proof follows closely the arguments in \cite{CC15}. To show \eqref{4.10a}, since $\beta\geq 0,\lambda\leq 0$ and $\phi_x\leq 0$, it suffices to estimate $\alpha \lambda\int_0^\infty h^2\omega\phi_xdx$.
Consider two non-negative functions $\varphi_1,\varphi_2\in C^{\infty}(0,\infty)$ with $\varphi_1+\varphi_2=1$ such that $\varphi_1$ is monotone non-increasing with $\varphi_1=0$ on $(\frac{1}{4},\infty)$, and $\varphi_2$ is monotone non-decreasing with $\varphi_2=0$ on $(\frac{1}{8},\infty)$.
With this partition of unity, we can write $\int_0^\infty h^2\omega\phi_x dx=\int_0^\infty h^2\omega\phi_x\varphi_1 dx+\int_0^\infty h^2\omega\phi_x\varphi_2 dx$. 
Using $0\leq-\phi_x\leq \phi$ on the support of $\varphi_1$ by the first property given in Lemma \ref{L.3.2}, we find
\[
-\int_0^\infty h^2\omega\phi_x\varphi_1 dx\leq \int_0^\infty h^2\omega\phi\, dx=I(t).
\]
Hence for all $\eps\in(0,1)$, 
\[
\alpha \lambda\int_0^\infty h^2\omega\phi_x\varphi_1 dx\leq \left(\eps+\frac{\alpha^2\lambda^2}{\eps}\right)I(t).
\]

For the other part with cut-off function $\varphi_2$, since $\int_0^\infty p(t,x)dx=1$,
\begin{align*}
-\int_0^\infty h^2\omega\phi_x\varphi_2 dx&= -\int_0^\infty h \partial_x\left(\int_0^x p(t,y)dy\right)\phi_x\varphi_2dx\\
&\leq \int_0^\infty |h_x|\phi_x\varphi_2dx+\int_0^\infty h|(\phi_x\varphi_2)_x|dx.
\end{align*}
Young's inequality and the fact that $\phi$ is supported in $(0,1)$ yield
\begin{align*}
-\alpha \lambda\int_0^\infty |h_x|\phi_x\varphi_2dx&\leq \eps\int_0^\infty h_x^2(\phi_x\varphi_2)^2dx+\frac{1}{\eps}\alpha^2 \lambda^2\\
&\leq {C\eps} J(t)    +\frac{1}{\eps}\alpha^2 \lambda^2.
\end{align*}
In the second inequality, we used $(\phi_x\varphi_2)^2\leq C\phi$ and $\omega(x)\leq \omega(1)$ on $(0,1)$. Similarly 
\begin{align*}
    -\alpha \lambda\int_0^\infty h|(\phi_x\varphi_2)_x|dx\leq {C\eps}I(t)+\frac{1}{\eps}\alpha^2 \lambda^2.
\end{align*}
Putting these estimates together proves \eqref{4.10a}.
As for \eqref{4.10b}, due to the third property stated in Lemma \ref{L.3.2}, we have
\begin{align*}
    \int_0^\infty h^2\omega \phi_{xx}dx&= \int_{x_0}^1 h^2\omega \phi_{xx}dx=\int_{x_0}^1 h\partial_x\left(\int_0^x p(t,y)dy\right)\phi_{xx}dx\\
    &=-\int_{x_0}^1 h_x\left(\int_0^x p(t,y)dy\right)\phi_{xx}dx-\int_{x_0}^1 h\left(\int_0^x p(t,y)dy\right)\phi_{xxx}dx\\
    &\leq \int_{x_0}^1 |h_x||\phi_{xx}|dx+\int_{x_0}^1 h|\phi_{xxx}|dx.
\end{align*}
where in the third equality $\phi_{xx}(x_0)=0$ is applied, and in the inequality, we used $\int_0^\infty p(t,y)dy=1$. Due to the second property of Lemma \ref{L.3.2} and $\omega>0$ on $[x_0,1]$, we obtain
\[
\int_0^\infty h^2\omega \phi_{xx}dx\leq C\int_0^\infty |h_x| \sqrt{\omega\phi}\, dx+C\int_{0}^\infty h\sqrt{\omega\phi}\,dx\leq C\eps (I(t)+J(t))+\frac{C}{\eps}
\]
which finishes the proof of \eqref{4.10b}.
\end{proof}






\begin{acks}[Acknowledgments]
The authors would like to thank two anonymous referees, an Associate
Editor and the Editor for their constructive comments that improved the
quality of this paper.
\end{acks}

\begin{funding}
E. Bayraktar is partially supported by the National Science Foundation under grant DMS-2106556 and by the Susan M. Smith chair.

W. Tang gratefully acknowledges financial support through an NSF grant DMS-2113779 and through a start-up grant at Columbia University.

Y.P. Zhang acknowledges partial support by an AMS-Simons Travel Grant.
\end{funding}

\end{document}